\documentclass[letterpaper]{article}

\usepackage{arxiv}

\usepackage{style}

\newcommand{\Hr}{{\rm H}}
\newcommand{\Wr}{{\rm W}}
\newcommand{\Lr}{{\rm L}}
\newcommand{\Cr}{{\rm C}}
\newcommand{\kt}{k_\text{int}}
\newcommand{\Th}{{{\cal T}_h}}

\usepackage{caption}

\title{Solving PDEs by Variational Physics-Informed Neural Networks:  an a posteriori error analysis}

\author{Stefano Berrone\thanks{Dipartimento di Scienze Matematiche, Politecnico di Torino, Corso Duca degli Abruzzi 24, 10129 Torino, Italy. stefano.berrone@polito.it (S. Berrone), claudio.canuto@polito.it (C. Canuto), moreno.pintore@polito.it (M. Pintore).}
	\And
	Claudio Canuto\footnotemark[1]
	\And Moreno Pintore\footnotemark[1]
}

\renewcommand{\shorttitle}{A posteriori error analysis for VPINNs}

\begin{document}

\maketitle

\begin{abstract}We consider the discretization of elliptic boundary-value problems by variational physics-informed neural networks (VPINNs), in which test functions are continuous, piecewise linear functions on a triangulation of the domain. We define an a posteriori error estimator, made of a residual-type term, a loss-function term, and data oscillation terms. We prove that the estimator is both reliable and efficient in controlling the energy norm of the error between the exact and VPINN solutions. Numerical results are in excellent agreement with the theoretical predictions.
\end{abstract}


\vspace{0.3cm}
\noindent \textbf{\textit{MSC-class}}  35A01, 65L10, 65L12, 65L20, 65L70

\section{Introduction} \label{sec1}

The possibility of using deep-learning tools for solving complex physical models has attracted the attention of many scientists over the last few years. We have in mind in this paper models that are mathematically described by partial differential equations, supplemented by suitable boundary and initial conditions. In the most general setting, if no information on the model is available except the knowledge of some of its solutions, the model may be completely surrogated by one or more neural network, trained by data (i.e., by the known solutions). However, in most situations of interest, the mathematical model is known (e.g., the Navier-Stokes equations describing an incompressible flow), and such information may be suitably exploited in training the network(s): one gets the so-called Physics Informed Neural Networks (PINNs). This approach was first proposed in \cite{raissi2019physics}, and it inspired further works such as e.g. \cite{tartakovsky2018learning} or \cite{yang2019adversarial}, until the recent paper \cite{LanthalerMishraKarniadakis2021} which presents a very general framework for the solution of operator equations by deep neural networks. PINNs are trained by using the strong form of the differential equations, which are enforced at a set of points in the domain by suitably defining the loss function. In this sense, PINNs can be viewed as particular instances of least-square/collocations methods.

Based on the weak formulation of the differential model, the so-called Variational Physics-Informed Neural Networks (VPINNs), proposed in \cite{kharazmi2019variational}, enforce the equations by means of suitably chosen test functions, not necessarily represented by neural networks \cite{khodayi2020varnet}; they are instances of least-square/Petrov-Galerkin methods. While the construction of the loss function is generally less expensive for PINNs than for VPINNs, the latter allow for the treatment of models with less regular solutions, as well as an easier enforcement of boundary conditions.  In addition, the error analysis for VPINNs takes advantage of the available results for the discretization of variational problems, in fulfilling the assumptions of Lax-Richmyer's theorem `stability plus consistency imply convergence'. Actually, consistency results follow rather easily from the recently established approximation properties of neural networks in Sobolev spaces (see, e.g., \cite{elbrachter2021deep}, \cite{guhring2020error}, \cite{opschoor2020deep}, \cite{kutyniok2021theoretical}, \cite{opschoor2021exponential}, \cite{gonon2021deep}), whereas the derivation of stability estimates for the neural network solution appears to be a less trivial task: indeed, a neural network is identified by its weights, which are usually much more than the conditions enforced in its training. In other words, the training of a neural network is functionally an ill-posed problem.

To this respect,  we considered in \cite{BeCaPi2021} a Petrov-Galerkin framework in which trial functions are defined by means of neural networks, whereas test functions are made of continuous, piecewise linear functions on a triangulation of the domain. Relying on an inf-sup condition between spaces of piecewise polynomial functions, we derived an a priori error estimate in the energy norm between the exact solution of an elliptic boundary-value problem and a high-order interpolant of a deep neural network, which minimizes the loss function. Numerical results indicate that the error follows a similar behavior when the interpolation operator is turned off. 

The purpose of the present paper is to perform an a posteriori error analysis for VPINNs, i.e., to get estimates on the error which only depend on the computed VPINN solution, rather than the unknown exact solution. This is important to get a practical and quantitative information on the quality of the approximation. After setting the model elliptic boundary-value problem in Sect. \ref{sec:setting}, and the corresponding VPINN discretization in Sect. \ref{sec:sub_discretization},  we define in Sect. \ref{sec:aposteriori-theory} a computable residual-type error estimator, and prove that it is both reliable and efficient in controlling the energy error between the exact solution and the VPINN solution. Reliability means that the global error is upper bounded by a constant times the estimator, efficiency means that the estimator cannot over-estimate the energy error, since the latter  is lower bounded by a constant times the former up to data oscillation terms. The proposed estimator is obtained by summing up several terms: one is the classical residual-type estimator in finite elements, measuring the bulk error inside each element of the triangulation as well as the inter-element gradient jumps; another term accounts for the magnitude of the loss function after minimization is performed; the remaining terms measure data oscillations, i.e., the errors committed by locally projecting the equation's coefficients and right-hand side upon suitable polynomial spaces. The estimator can be written as a sum of elemental contributions, thereby allowing its use within an adaptive discretization strategy which refines the elements carrying the largest contributions to the estimator.

\section{The model boundary-value problem} \label{sec:setting}

Let $\Omega \subset \mathbb{R}^n$ be  a bounded polygonal/polyhedral domain with Lipschitz boundary $\Gamma=\partial\Omega$.

Let us consider the model elliptic boundary-value problem
\begin{equation}\label{eq:model-pb}
\begin{cases}
Lu:=-\nabla \cdot (\mu \nabla u) + \boldsymbol{\beta}\cdot \nabla u + \sigma u =f & \text{in \ } \Omega\,, \\
u=0 & \text{on \ } \Gamma \,, \end{cases}
\end{equation}
where $\mu, \sigma \in \Lr^\infty(\Omega)$, $ \boldsymbol{\beta} \in (\Wr^{1,\infty}(\Omega))^n$ satisfy $\mu \geq \mu_0$, $\sigma - \frac12 \nabla \cdot \boldsymbol{\beta} \geq 0$ in $\Omega$ for some constant $\mu_0>0$, whereas $f \in L^2(\Omega)$.

\smallskip
Setting $V=\Hr^1_{0}(\Omega)$,  define the bilinear and linear forms 
\begin{equation}\label{eq:form a}
a:V\times V \to \mathbb{R}\,, \qquad a(w,v)=\int_\Omega \mu \nabla w \cdot \nabla v + \boldsymbol{\beta}\cdot \nabla w \, v + \sigma w \, v\,,
\end{equation}
\begin{equation}\label{eq:forms F}
F:V\to \mathbb{R}\,, \qquad F(v)=\int_\Omega f \, v  \,;
\end{equation}
denote by $\alpha \geq \mu_0$  the coercivity constant of the form $a$, and by $\Vert a \Vert$, $\Vert F \Vert$ the continuity constants of the forms $a$ and $F$. Problem \eqref{eq:model-pb} is formulated variationally as follows: {\it Find $u \in V $ such that}
\begin{equation}\label{eq:model-pb-var}
a(u,v)=F(v) \qquad \forall v \in V\,.
\end{equation}

\begin{remark}[Other boundary conditions]\label{rem:other-bcs}
{\rm
The forthcoming formulation of the discretized problem and the a posteriori error analysis can be extended without pain to cover the case of mixed Dirichlet-Neumann boundary conditions, namely $u=g$ on $\Gamma_D$, $\mu \partial_n u =\psi$ on $\Gamma_N$, with $\Gamma_D \cup \Gamma_N = \Gamma$. We just consider homogeneous Dirichlet conditions to avoid an excess of technicalities. 
}
\end{remark}

\subsection{The VPINN discretization} \label{sec:sub_discretization}
We aim at approximating the solution of Problem \eqref{eq:model-pb} by a generalized Petrov-Galerkin strategy. 

To define the subset of $V$ of trial functions, let us choose a fully-connected feed-forward neural network structure $\NN$, with $n$ input variables and 1 output variable, identified by the number of layers $L$, the layer widths $N_\ell$,  $\ell=1, \dots, L$, and the activation function $\rho$. Thus, each choice of the weights ${\mathbf w} \in \mathbb{R}^N$ defines a mapping $w^\NN :  \boldsymbol{x} \mapsto w(\boldsymbol{x},{\mathbf w})$, which we think as restricted to the closed domain $\bar{\Omega}$; let us denote by $W^\NN$ the manifold containing all functions that can be generated by this neural network structure. We enforce the homogeneous Dirichlet boundary conditions by multiplying each $w$ by a fixed smooth function $\Phi \in V$ (we refer to \cite{sukumar2022} for a general strategy to construct this function); we assume that $v^\NN = \Phi w^\NN$ belongs to $V$ for any $w^\NN \in W^\NN$. In conclusion, our manifold of trial functions will be
$$
V^\NN = \{ v^\NN \in V : v^\NN=\Phi w^\NN \text{ for some }w^\NN \in W^\NN \}\,.
$$

To define the subspace of $V$ of test functions,  let us introduce  a conforming, shape-regular triangulation ${\cal T}_h= \{ E \}$ of $\bar{\Omega}$ with meshsize $h>0$ and let $V_h \subset V$ be the linear subspace formed by the functions which are piecewise linear polynomials  over the triangulation ${\cal T}_h$. Furthermore, let us introduce computable approximations of the forms $a$ and $F$ by numerical quadratures. Precisely, for any $E \in {\cal T}_h$, let $\{(\xi^E_\iota,\omega^E_\iota) : \iota \in I^E\}$ be the nodes and weights of a quadrature formula of precision $q \geq 2$  
on $E$. Then, assuming that all data $\mu$, $\boldsymbol{\beta}$, $\sigma$, $f$ are continuous in each element of the triangulation, we define the approximate forms
\begin{equation}\label{eq:def-ah}
a_h(w,v)= \sum_{E \in {\cal T}_h} \sum_{\iota \in I^E} [\mu \nabla w \cdot \nabla v + \boldsymbol{\beta}\cdot \nabla w \, v + \sigma w v](\xi^E_\iota) \,\omega^E_\iota\,, 
\end{equation}
\begin{equation}\label{eq:def-Fh}
F_h(v) =  \sum_{E \in {\cal T}_h}  \sum_{\iota \in I^E} [ f v](\xi^E_\iota) \,\omega^E_\iota \,.
\end{equation}

With these ingredients at hand, we would like to approximate the solution of Problem \eqref{eq:model-pb-var} by some $u^{\cal N\!N} \in V^{\cal N\!N}$ satisfying
\begin{equation}\label{eq:PGproblem}
a_h(u^{\cal N\!N},v_h)=F_h(v_h) \qquad \forall v_h \in V_h\,. 
\end{equation}

In order to handle this problem by the neural network, let us introduce a basis in $V_h$, say $V_h = \text{span}\{\varphi_i : i\in I_h\}$, and for any $w \in V $ let us define the residuals
\begin{equation}\label{eq:residuals}
r_{h,i}(w)=F_h(\varphi_i)-a_h(w,\varphi_i)\,, \qquad i \in I_h\,,
\end{equation}
as well as the loss function
\begin{equation}\label{eq:loss-function}
R_h^2(w) = \sum_{i \in I_h} r_{h,i}^2(w) \,. 
\end{equation}
Then, we search for a global minimum of the loss function in $V^{\cal N\!N}$, i.e., we consider the following minimization problem:  {\it Find $u^{\cal N\!N} \in V^{\cal N\!N}$ such that}
\begin{equation}\label{eq:min-prob}
u^{\cal N\!N} \in \displaystyle{\text{arg}\!\!\!\!\min_{w \in V^{\cal N\!N}}}\, R_h^2(w) \,.
\end{equation}

Note that any solution $u^{\cal N\!N}$ of \eqref{eq:PGproblem} annihilates the loss function, hence it is a solution of \eqref{eq:min-prob}; such a solution may not be unique, since the set of equations \eqref{eq:PGproblem} may be underdetermined (in particular, for $f=0$ one may obtain a non-zero $u^{\cal N\!N}$, see \cite[Sect. 6.3]{BeCaPi2021}). On the other hand, system \eqref{eq:PGproblem} may be overdetermined, and admit no solution; in this case, the loss function will have strictly positive minima.

\begin{remark}[Discretization with interpolation] \label{rem:no-interp}{\rm
In order to reduce and control the randomic effects related to the use of a network depending upon a large number of weights, in \cite{BeCaPi2021} we proposed to locally project the neural network upon a space of polynomials,  before computing the loss function. 

To be precise, we have considered a conforming, shape-regular partition ${\cal T}_H=\{G\}$ of $\bar{\Omega}$, which is equal to or coarser than ${\cal T}_h$ (i.e., each element $E \in {\cal T}_h$ is contained in an element $G \in {\cal T}_H$) but compatible with ${\cal T}_h$ (i.e., its meshsize $H>0$ satisfies $H\lesssim h$). Let $V_H \subset V$ be the linear subspace formed by the functions which are  piecewise polynomials of degree $\kt=q+1$ over the triangulation ${\cal T}_H$, and let ${\cal I}_H : \Cr^0(\bar{\Omega}) \to V_H$ be the associated element-wise Lagrange interpolation operator. 

Given a neural network $w \in V^\NN$, let us denote by $w_H= {\cal I}_H w^\NN \in V_H$ its piecewise polynomial interpolant.
Then, the definition \eqref{eq:residuals} of local residuals is modified as 
\begin{equation}\label{eq:residuals-tilde}
\tilde{r}_{h,i}(w)=F_h(\varphi_i)-a_h(w_H,\varphi_i)\,, \qquad i \in I_h\,;
\end{equation}
consequently, the loss function takes the form
\begin{equation}\label{eq:loss-function-tilde}
\tilde{R}_h^2(w) = \sum_{i \in I_h} \tilde{r}_{h,i}^2(w) \,,
\end{equation}
and we define a new approximation of the solution of Problem \eqref{eq:model-pb-var} by setting
\begin{equation}\label{eq:min-prob-tilde}
\tilde{u}^\NN_H = {\cal I}_H \tilde{u}^\NN \in V_H\,, \qquad \text{where} \quad \tilde{u}^{\cal N\!N} \in \displaystyle{\text{arg}\!\!\!\!\min_{w \in U^{\cal N\!N}}}\, \tilde{R}_h^2(w) \,.
\end{equation}
In \cite{BeCaPi2021} we derived an a priori error estimate for the error $\Vert u - \tilde{u}^\NN_H \Vert_V$, and we documented the error decay as $h \to\infty$, which turns out to have a more regular behavior that the error $\Vert u - {u}^\NN \Vert_V$, although the latter is usually smaller. 

The subsequent a posteriori error analysis could be extended to give a control on the error produced by $\tilde{u}^\NN_H $ as well. For the sake of simplicity, we do not pursue such a task here.
}
\end{remark}

\section{The a posteriori error estimator}\label{sec:aposteriori-theory}


In order to build an error estimator, let us first choose, for any $E \in {\cal T}_h$ and any $k \geq 0$, a projection operator $\Pi_{E,k} : L^2(E) \to \mathbb{P}_k(E)$ satisfying 
\begin{equation}\label{eq:mean}
\int_E \Pi_{E,k} \varphi = \int_E \varphi \qquad \forall \varphi \in L^2(E) \,.
\end{equation}
This allows us to introduce approximate bilinear and linear forms
\begin{equation}\label{eq:form aPi}
a_\pi(w,v)=\sum_{E \in {\cal T}_h} \int_E \Pi_{E,q}\left( \mu \nabla w \right) \cdot \nabla v + \Pi_{E,q-1}\left( \boldsymbol{\beta}\cdot \nabla w + \sigma w\right)  v\,,
\end{equation}
\begin{equation}\label{eq:forms FPi}
 F_\pi (v)=\sum_{E \in {\cal T}_h} \int_E \left(\Pi_{E,q-1} f\right)  v  \,,
\end{equation}
which are useful in the forthcoming derivation. Indeed, the coercivity of the form $a$ allows us to bound the $V$-norm of the error as follows:
\begin{equation}\label{eq:inf-sup}
\vert u - u^\NN \vert_{1,\Omega} \leq \frac1\alpha \sup_{v \in V} \frac{a(u - u^\NN,v)}{\vert v \vert_{1,\Omega}} \,.
\end{equation}
We split the numerator as
\begin{equation}\label{eq:split-a}
\begin{split}
a(u - u^\NN,v) &= F(v) -a(u^\NN,v) 
= \underbrace{F(v)-F_\pi(v)}_{(\text{I})}  \ + \ \underbrace{F_\pi(v)-a_\pi(u^\NN,v)}_{(\text{III})}\\
& \quad  + \ \underbrace{a_\pi(u^\NN,v) - a(u^\NN,v)}_{(\text{II})} 
\end{split}
\end{equation}
and we proceed to bound each term on the right-hand side.

The terms $({\rm I})$ and  $({\rm II})$ account for the element-wise projection error upon polynomial spaces; they are estimated in the next two Lemmas.  

\begin{lemma}\label{lem:bound-I}
The quantity $({\rm I})$ defined in \eqref{eq:split-a} satisfies
\begin{equation}\label{eq:bound-I}
\vert  ( {\rm I} ) \vert  \lesssim \Big(\sum_{E \in {\cal T}_h} \eta_{{\rm rhs},1}^2(E) \Big)^{1/2} \vert v \vert_{1,\Omega}\,, 
\end{equation}
with 
\begin{equation}\label{eq:eta-f}
\eta_{{\rm rhs},1}(E) = h_E \Vert f - \Pi_{E,q-1} f \Vert_{0,E} \,.
\end{equation}
\end{lemma}
\proof Setting $m_E(v)=\frac1{\vert E \vert} \int_E v$ and using \eqref{eq:mean}, we get
$$
( {\rm I} ) = \sum_{E \in {\cal T}_h} \int_E \left( f - \Pi_{E,q-1} f \right)(v-m_E(v) )  \,,
$$
and we conclude using the bound $\Vert v - m_E(v) \Vert_{0,E} \lesssim h_E \vert v \vert_{1,E}$. \endproof

\begin{lemma}\label{lem:bound-III}
The quantity $({\rm II})$ defined in \eqref{eq:split-a} satisfies
\begin{equation}\label{eq:bound-III}
\vert  ( {\rm II} )  \vert  \lesssim \Big( \sum_{E \in {\cal T}_h}  \big( \eta_{{\rm coef},1}^2(E) + \eta_{{\rm coef},2}^2(E) + \eta_{{\rm coef},3}^2(E) \big)
 \Big)^{1/2} \vert v \vert_{1,\Omega}\,, 
\end{equation}
with
\begin{equation}\label{eq:eta-coef-13}
\begin{split}
\eta_{{\rm coef},1}(E)  &= \Vert  \mu \nabla u^\NN - \Pi_{E,q} (\mu \nabla u^\NN) \Vert_{0,E} \,, \\[3pt]
\eta_{{\rm coef},2}(E) &= h_E \Vert  \boldsymbol{\beta}\cdot \nabla u^\NN - \Pi_{E,q-1}( \boldsymbol{\beta}\cdot \nabla u^\NN)  \Vert_{0,E} \,, \\[3pt]
\eta_{{\rm coef},3}(E) &= h_E \Vert  \sigma u^\NN - \Pi_{E,q-1}( \sigma u^\NN)  \Vert_{0,E} \,.
\end{split}
\end{equation}
\end{lemma}
\proof It holds
\begin{equation*}
\begin{split}
({\rm II}) &=  \sum_{E \in {\cal T}_h} \int_E \Big( \mu \nabla u^\NN - \Pi_{E,q}(\mu \nabla u^\NN) \Big) \cdot \nabla v \\
& \quad + \sum_{E \in {\cal T}_h} \int_E \Big(  \boldsymbol{\beta}\cdot \nabla u^\NN - \Pi_{E,q-1}( \boldsymbol{\beta}\cdot \nabla u^\NN)  \Big) (v - m_E(v)) \\
& \quad + \sum_{E \in {\cal T}_h} \int_E \Big( \sigma u^\NN - \Pi_{E,q-1}( \sigma u^\NN)  \Big) (v - m_E(v)) \,,
\end{split}
\end{equation*}
where we have used again \eqref{eq:mean}. We conclude as in the proof of Lemma \ref{lem:bound-I}.\endproof

Let us now focus on the quantity  $({\rm III})$, which can be written as
\begin{equation}\label{eq:split-III}
({\rm III}) = \underbrace{F_\pi(v-v_h) - a_\pi(u^\NN,v-v_h)}_{(\text{IV})}  +  \underbrace{F_\pi(v_h) - a_\pi(u^\NN,v_h)}_{(\text{V})} \,, \qquad \forall v_h \in V_h\,;  
\end{equation}
in turn, the quantity  $({\rm V})$ can be written as
\begin{equation}\label{eq:split-V}
({\rm V}) = \underbrace{F_\pi(v_h) -F_h(v_h)}_{(\text{VII})} + \underbrace{F_h(v_h)-a_h(u^\NN,v_h)}_{(\text{VI})} + \underbrace{a_h(u^\NN,v_h)-a_\pi(u^\NN,v_h)}_{(\text{VIII})} \,. 
\end{equation}

The bound of $({\rm IV})$ is standard in finite-element a posteriori error analysis: it involves the local bulk residuals

\begin{equation}\label{eq:def-bulk}
{\rm bulk}_E(u^\NN) = \Pi_{E,q-1}f +\nabla \cdot \Pi_{E,q} (\mu \nabla u^\NN) - \Pi_{E,q-1}( \boldsymbol{\beta}\cdot \nabla u^\NN + \sigma u^\NN)
\end{equation}
and the interelement jumps at each edge $e$ shared by two elements, say $E_1$ and $E_2$ with opposite normal unit vectors $\boldsymbol{n}_1$ and $\boldsymbol{n}_2$, namely

\begin{equation}\label{eq:def-jump}
{\rm jump}_e(u^\NN) = \Pi_{E_1,q}(\mu \nabla u^\NN)\cdot \boldsymbol{n}_1 + \Pi_{E_2,q}(\mu \nabla u^\NN)\cdot \boldsymbol{n}_2
\,;
\end{equation}
in addition, one defines ${\rm jump}(u^\NN, e) =0$ if $e \subset \partial \Omega$. 

To derive the bound, the test function $v_h$ in \eqref{eq:split-III} is chosen as $v_h=I_h^C v$, the Cl\'ement interpolant of $v$ on $\Th$ \cite{clement1975}, which satisfies 
\begin{equation}\label{eq:clement}
\Vert v-I_h^C v \Vert_{k,E} \lesssim h_E^k \vert v \vert_{1, D_E}, \qquad k=0,1 \,,
\end{equation}
where $D_E = \cup \{E' \in \Th : E \cap E' \not= \emptyset\}$. 

\begin{lemma}\label{lem:bound-IV}
The quantity $({\rm IV})$ defined in \eqref{eq:split-III} satisfies
\begin{equation}\label{eq:bound-IV}
\vert  ( {\rm IV} ) \vert  \lesssim \Big(\sum_{E \in {\cal T}_h} \eta_{{\rm res}}^2(E) \Big)^{1/2} \vert v \vert_{1,\Omega}\,, 
\end{equation}
where 
\begin{equation}\label{eq:eta-res}
\eta_{{\rm res}}(E) = h_E \Vert \, {\rm bulk}_E(u^\NN)  \, \Vert_{0,E}  +  h_E^{1/2} \sum_{e \subset \partial E} \Vert \,{\rm jump}_e(u^\NN)   \, \Vert_{0,e} \,,
\end{equation}
with ${\rm bulk}_E(u^\NN)$ defined in \eqref{eq:def-bulk} and ${\rm jump}_e(u^\NN)$ defined in \eqref{eq:def-jump}. 
\end{lemma}
\begin{proof} We refer e.g. to \cite{verfurth1996} for more details. 
\end{proof}  

Before considering the  quantity $({\rm VI})$, let us state a useful result of equivalence of norms.
\begin{lemma}\label{lem:equi-norm}
For any $v_h =  \sum_{i \in I_h} v_i \varphi_i \in V_h$, let $\boldsymbol{v} = (v_i)_{i \in I_h}$ be the vector of its coefficients. There exist 
constants $0< c_h \leq C_h$, possibly depending on $h$ such that
\begin{equation}\label{eq:norm-equiv-Vh}
c_h \vert v_h \vert_{1, \Omega} \leq \Vert \boldsymbol{v} \Vert_2 \leq C_h \vert v_h \vert_{1, \Omega} \qquad \forall v_h \in V_h \,,
\end{equation}
where $\Vert \boldsymbol{v} \Vert_2 = \left( \sum_{i \in I_h} v_i^2 \right)^{1/2}$.
\end{lemma} 
\begin{proof} The result expresses the equivalence of norms in finite dimensional spaces. If the triangulation ${\cal T}_h$ is quasi uniform, then one can prove by a standard reference-element argument that $c_h \simeq h^{1-n/2}$ whereas $C_h \simeq h^{-n/2}$.
\end{proof}

We are now able to bound the quantity $({\rm VI})$ in terms of the loss function introduced in \eqref{eq:loss-function}, as follows. 
\begin{lemma}\label{lem:bound-VI}
The quantity $({\rm VI})$ defined in \eqref{eq:split-V} satisfies
\begin{equation}\label{eq:bound-VI}
\vert  ( {\rm VI} ) \vert  \lesssim \eta_{{\rm loss}} \vert v \vert_{1,\Omega}\,, 
\end{equation}
where
\begin{equation}
\eta_{{\rm loss}} = C_h R_h(u^\NN)
\end{equation}
and the constant $C_h$ is defined in  \eqref{eq:norm-equiv-Vh}.  
\end{lemma}
\begin{proof}
Writing $v_h =  \sum_{i \in I_h} v_i \varphi_i$, it holds
$$
({\rm VI}) = \sum_{i \in I_h} r_{h,i}(u^\NN) v_i \,,
$$
whence
$$
\vert ({\rm VI}) \vert \lesssim R_h(u^\NN) \Vert \boldsymbol{v} \Vert_2 \,,
$$
We conclude by using \eqref{eq:norm-equiv-Vh} and observing that  
\begin{equation}\label{eq:clem-bound}
\vert v_h \vert_{1, \Omega} \lesssim  \vert v \vert_{1, \Omega} \,,
\end{equation}
since we have chosen $v_h=I_h^C v$ and \eqref{eq:clement} holds.
\end{proof}

We are left with the problem of bounding the terms $({\rm VII})$ and $({\rm VIII})$ in \eqref{eq:split-V}. They are similar to the terms $({\rm I})$ and $({\rm II})$, respectively, but reflect the presence of the quadrature formula introduced in \eqref{eq:def-ah} and \eqref{eq:def-Fh}.
In the forthcoming analysis, it will be useful to introduce the following notation for the quadrature-based discrete (semi-)norm on $C^0 (E)$:
\begin{equation}\label{eq:discr-norm}
\Vert \varphi \Vert_{0, E, \omega}= \left(\sum_{\iota \in I^E} \varphi^2(\xi^E_\iota) \,\omega^E_\iota \right)^{1/2} \,.
\end{equation}

Let us start with the quantity $({\rm VII})$. Recalling that the adopted quadrature rule has precision $q$ and test functions $v_h$ are piecewise linear polynomials, it holds
\begin{equation}\label{eq:split-VII}
\begin{split}
({\rm VII}) &= \sum_{E \in {\cal T}_h} \left( \int_E  (\Pi_{E,q-1} f) v_h - \sum_{\iota \in I^E} f(\xi^E_\iota) v_h(\xi^E_\iota) \,\omega^E_\iota \right) \\
&= \sum_{E \in {\cal T}_h} \left( \sum_{\iota \in I^E}  (\Pi_{E,q-1} f - f) (\xi^E_\iota) v_h(\xi^E_\iota) \,\omega^E_\iota \right) \\
&= \underbrace{\sum_{E \in {\cal T}_h} \left( \sum_{\iota \in I^E}  (\Pi_{E,q-1} f - f) (\xi^E_\iota) (v_h - m_E(v_h))(\xi^E_\iota) \,\omega^E_\iota \right)}_{(\text{VIIa})}  \\
& \qquad + \underbrace{\sum_{E \in {\cal T}_h} \left( \sum_{\iota \in I^E}  (\Pi_{E,q-1} f - f) (\xi^E_\iota) \,\omega^E_\iota  m_E(v_h) \right)}_{(\text{VIIb})} \,.
\end{split}
\end{equation}
On the one hand, recalling the assumption $q \geq 2$ and inequality \eqref{eq:clem-bound} one has
\begin{equation}
\begin{split}
\vert ({\rm VIIa}) \vert & \leq \sum_{E \in {\cal T}_h} \Vert f-\Pi_{E,q-1} f \Vert_{0, E, \omega} \Vert v_h - m_E(v_h) \Vert_{0, E, \omega} \\
& = \sum_{E \in {\cal T}_h} \Vert f-\Pi_{E,q-1} f \Vert_{0, E, \omega} \Vert v_h - m_E(v_h) \Vert_{0, E} \\
&  \lesssim \sum_{E \in {\cal T}_h}  h_E \Vert f-\Pi_{E,q-1} f \Vert_{0, E, \omega} \vert v_h \vert_{1,E} \\
& \lesssim \left( \sum_{E \in {\cal T}_h}  h_E^2 \Vert f-\Pi_{E,q-1} f \Vert_{0, E, \omega}^2 \right)^{1/2}  \vert v \vert_{1,\Omega} \,.
 \end{split}
\end{equation}
On the other hand, we first observe that, by the exactness of the quadrature rule and \eqref{eq:mean}, we get
$$
\sum_{\iota \in I^E}  (\Pi_{E,q-1} f )(\xi^E_\iota) \,\omega^E_\iota = \int_E \Pi_{E,q-1} f = \int_E f = \int_E \Pi_{E,q} f  = 
\sum_{\iota \in I^E}  (\Pi_{E,q} f )(\xi^E_\iota) \,\omega^E_\iota.
$$
Hence,
\begin{equation}
\begin{split}
\vert ({\rm VIIb}) \vert & \leq \sum_{E \in {\cal T}_h} \Vert f-\Pi_{E,q} f \Vert_{0, E, \omega} \Vert m_E(v_h) \Vert_{0, E} \\
&  \leq \sum_{E \in {\cal T}_h}  \Vert f-\Pi_{E,q} f \Vert_{0, E, \omega} \Vert v_h \Vert_{0,E} \\
& \lesssim \left( \sum_{E \in {\cal T}_h}  \Vert f-\Pi_{E,q} f \Vert_{0, E, \omega}^2 \right)^{1/2}  \vert v \vert_{1,\Omega} \,.
 \end{split}
\end{equation}
Summarizing, we obtain the following result, which is anologous to that in Lemma \ref{lem:bound-I}.
\begin{lemma}\label{lem:bound-VII}
The quantity $({\rm VII})$ defined in \eqref{eq:split-V} satisfies
\begin{equation}\label{eq:bound-VII}
\vert  ( {\rm VII} ) \vert  \lesssim \Big(\sum_{E \in {\cal T}_h} \eta_{{\rm rhs},2}^2(E) \Big)^{1/2} \vert v \vert_{1,\Omega}\,, 
\end{equation}
with 
\begin{equation}\label{eq:eta-f-2}
\eta_{{\rm rhs},2}(E) = h_E \Vert f - \Pi_{E,q-1} f \Vert_{0,E,\omega} + \Vert f - \Pi_{E,q} f \Vert_{0,E,\omega} \,.
\end{equation}
\end{lemma}

The last term in \eqref{eq:split-V}, $({\rm VIII})$, can be written as
\begin{equation}\label{eq:split-VIII}
\begin{split}
({\rm VIII}) &= \underbrace{\sum_{E \in {\cal T}_h} \left( \sum_{\iota \in I^E} (\mu \nabla u^\NN)(\xi^E_\iota) \cdot \nabla v_h \,\omega^E_\iota  -  \int_E  \Pi_{E,q}(\mu \nabla u^\NN)\cdot \nabla v_h  \right)}_{(\text{VIIIa})} \\
& \ \ + \underbrace{\sum_{E \in {\cal T}_h} \left( \sum_{\iota \in I^E} ( \boldsymbol{\beta}\cdot \nabla u^\NN)(\xi^E_\iota) \, v_h(\xi^E_\iota) \,\omega^E_\iota - \int_E  \Pi_{E,q-1}(\boldsymbol{\beta}\cdot \nabla u^\NN) \, v_h \right)}_{(\text{VIIIb})} \\
& \ \ + \underbrace{\sum_{E \in {\cal T}_h} \left( \sum_{\iota \in I^E} (\sigma u^\NN)(\xi^E_\iota)\,  v_h (\xi^E_\iota) \,\omega^E_\iota - \int_E  \Pi_{E,q-1}(\sigma u^\NN) \, v_h \right)}_{(\text{VIIIc})} \,.
\end{split}
\end{equation}
Concerning $(\text{VIIIa})$, by the exactness of the quadrature rule and the fact that $\nabla v_h$ is piecewise constant, one has
$$
(\text{VIIIa}) = \sum_{E \in {\cal T}_h}   \sum_{\iota \in I^E} \big(\mu \nabla u^\NN - \Pi_{E,q}(\mu \nabla u^\NN)\big)(\xi^E_\iota) \cdot \nabla v_h \,\omega^E_\iota \,,
$$
which easily gives
$$
\vert  ( {\rm VIIIa} ) \vert  \lesssim \left( \sum_{E \in {\cal T}_h}  \Vert \mu \nabla u^\NN - \Pi_{E,q}(\mu \nabla u^\NN) \Vert_{0,E,\omega}^2 \right)^{1/2} \vert v \vert_{1,\Omega} \,.
$$
The terms $(\text{VIIIb})$ and $(\text{VIIIc})$ are similar to the term $(\text{VII})$ above, in which $f$ is replaced by $\boldsymbol{\beta}\cdot \nabla u^\NN$ and $\sigma u^\NN$, respectively. Hence, they can be bounded as done for $(\text{VII})$. Summarizing, we obtain the following result, which is anologous to that in Lemma \ref{lem:bound-III}.
\begin{lemma}\label{lem:bound-VIII}
The quantity $({\rm VIII})$ defined in \eqref{eq:split-V} satisfies
\begin{equation}\label{eq:bound-VIII}
\vert  ( {\rm VIII} )  \vert  \lesssim \Big( \sum_{E \in {\cal T}_h}  \big( \eta_{{\rm coef},4}^2(E) + \eta_{{\rm coef},5}^2(E) + \eta_{{\rm coef},6}^2(E) \big)
 \Big)^{1/2} \vert v \vert_{1,\Omega}\,, 
\end{equation}
with
\begin{equation}\label{eq:eta-coef-46}
\begin{split}
\eta_{{\rm coef},4}(E)  &= \Vert  \mu \nabla u^\NN - \Pi_{E,q} (\mu \nabla u^\NN) \Vert_{0,E,\omega} \,, \\[3pt]
\eta_{{\rm coef},5}(E) &= h_E \Vert  \boldsymbol{\beta}\cdot \nabla u^\NN - \Pi_{E,q-1}( \boldsymbol{\beta}\cdot \nabla u^\NN)  \Vert_{0,E,\omega} \,, \\[3pt]
& \qquad \qquad \qquad + \Vert  \boldsymbol{\beta}\cdot \nabla u^\NN - \Pi_{E,q}( \boldsymbol{\beta}\cdot \nabla u^\NN)  \Vert_{0,E,\omega} \\[3pt]
\eta_{{\rm coef},6}(E)  &= h_E \Vert  \sigma u^\NN - \Pi_{E,q-1}( \sigma u^\NN)  \Vert_{0,E,\omega} \\[3pt]
& \qquad \qquad \qquad + \Vert  \sigma u^\NN - \Pi_{E,q}( \sigma u^\NN)  \Vert_{0,E,\omega}\,.
\end{split}
\end{equation}
\end{lemma}

At this point, we are ready to derive the announced a posteriori error estimates. In order to get an upper bound of the error, we concatenate \eqref{eq:inf-sup}, \eqref{eq:split-a}, \eqref{eq:split-III}, \eqref{eq:split-V}, and use the bounds given in Lemmas  \ref{lem:bound-I} to \ref{lem:bound-VIII}, arriving at the following result.

\begin{theorem}[a posteriori upper bound of the error]\label{theo:aposteriori-up}
Let $u^{\cal N\!N} \in V^{\cal N\!N}$ satisfy \eqref{eq:min-prob}. Then, the error $u-u^{\cal N\!N} $ can be estimated from above as follows:
\begin{equation}\label{eq:aposteriori1}
\vert u - u^\NN \vert_{1,\Omega} \lesssim  \left( \eta_{\rm res}  + \eta_{\rm loss} + \eta_{\rm coef} + \eta_{\rm rhs} \right) \,,
\end{equation}
where
\begin{equation}
\begin{split}
\eta_{\rm res}^2 &= \sum_{E \in {\cal T}_h} \eta_{\rm res}^2(E) \,, \quad  \eta_{\rm coef}^2 = \sum_{E \in {\cal T}_h} \sum_{k=1}^6\eta_{{\rm coef},k}^2(E)  \,, \quad  \eta_{\rm rhs}^2 = \sum_{E \in {\cal T}_h} \sum_{k=1}^2\eta_{{\rm rhs},k}^2(E)\,.
\end{split}
\end{equation}
\end{theorem}

We realize that the global estimator $\eta= \eta_{\rm res}  + \eta_{\rm loss} + \eta_{\rm coeff} + \eta_{\rm rhs}$ is the sum of four contributions: $\eta_{\rm res}$ is the classical residual-based estimator, $\eta_{\rm loss}$ measures how small the minimized loss function is, i.e., how well the discrete variational equations \eqref{eq:PGproblem} are fulfilled, whereas $\eta_{\rm coef}$ and $\eta_{\rm rhs}$ reflect the error in approximating elementwise the coefficients of the operator and the right-hand side by polynomials of degrees related to the precision of the quadrature formula.

\medskip

It is possible to derive from \eqref{eq:aposteriori1} an element-based a posteriori error estimator, which can be used to design an adaptive strategy of mesh refinement (see, e.g. \cite{NochettoCIME2012}). To this end, from now on we assume that the basis $\{\varphi_i : i\in I_h\}$ of $V_h$, introduced to define \eqref{eq:residuals}, is the canonical Lagrange basis associated with the nodes of the triangulation $\Th$. 
Given any $E \in \Th$, we introduce the elemental index set $I_h^E =\{ i \in I_h : E \subset {\rm supp}\, \varphi_i\}$, where 
${\rm supp}\, \varphi_i$ is the support of $\varphi_i$, and we define a local contribution to the term $\eta_{\rm loss}$ as follows:
\begin{equation}\label{eq:eta-local-loss}
\eta_{\rm loss}^2(E) = C_h^2 \sum_{i \in I_h^E}  r_{h,i}^2(u^\NN) \,,
\end{equation}
which satisfies 
$$
\eta_{\rm loss}^2 \leq  \sum_{E \in {\cal T}_h} \eta_{\rm loss}^2(E) \,.
$$
With this definition at hand, we can introduce the following elemental error estimator.
\begin{definition}[elemental error estimator]\label{def:error-est}
For any $E \in {\cal T}_h$, let us set
\begin{equation}\label{eq:aposteriori3}
\eta^2(E)  = \eta_{\rm res}^2(E)  + \eta_{\rm loss}^2(E) + \sum_{k=1}^6\eta_{{\rm coef},k}^2(E)  +  \sum_{k=1}^2\eta_{{\rm rhs},k}^2(E) \,,
\end{equation}
where the addends in this sum are defined, respectively, in \eqref{eq:eta-res}, \eqref{eq:eta-local-loss}, \eqref{eq:eta-coef-13} and \eqref{eq:eta-coef-46}, \eqref{eq:eta-f} and \eqref{eq:eta-f-2}.
\end{definition}

Then, Theorem \ref{theo:aposteriori-up} can be re-formulated in terms of these quantities.
\begin{corollary}[localized a posteriori error estimator]\label{cor:aposteriori}
The error $u-u^{\cal N\!N} $ can be estimated as follows:
\begin{equation}\label{eq:aposteriori2}
\vert u - u^\NN \vert_{1,\Omega} \lesssim  \Big( \sum_{E \in {\cal T}_h}  \eta^2(E) \Big)^{1/2} \,.
\end{equation}
\end{corollary}

Inequality \eqref{eq:aposteriori2} guarantees the {\em reliability} of the proposed error estimator, namely the estimator does provide a computable upper bound of the discretization error. Next result assures that the estimator is also {\em efficient}, namely it does not overestimate the error.
\begin{theorem}[a posteriori lower bound of the error]\label{theo:aposteriori-down}
Let $u^{\cal N\!N} \in V^{\cal N\!N}$ satisfy \eqref{eq:min-prob}. Then, the error $u-u^{\cal N\!N} $ can be locally estimated from below as follows: for any $E \in {\cal T}_h$ it holds
\begin{eqnarray}\label{eq:aposteriori3a}
\eta_{{\rm res}}(E)   &\lesssim& \vert u - u^\NN \vert_{1,D_E} + \sum_{E' \subset D_E} \left(
 \sum_{k=1}^3 \eta_{{\rm coef},k}^2(E') + \eta_{{\rm rhs},1}^2(E') \right)^{1/2} \!\!\!\!\!\!\!\!\,, \\ \label{eq:aposteriori3}
\frac{c_h}{C_h} \, \eta_{{\rm loss}}(E) &\lesssim & \vert u - u^\NN \vert_{1,D_E} + \sum_{E' \subset D_E} \left(
 \sum_{k=1}^6 \eta_{{\rm coef},k}^2(E') + \sum_{k=1}^2\eta_{{\rm rhs},k}^2(E') \right)^{1/2}   \label{eq:aposteriori3b} 
\end{eqnarray}
\end{theorem}
\proof
To derive \eqref{eq:aposteriori3a}, let us first consider the bulk contribution to the estimator. We apply a classical argument in a posteriori analysis, namely we introduce a non-negative bubble function $b_E \in V$ with support in $E$ and such that $\Vert \phi \Vert_{0,E} \simeq \Vert  b_E^{1/2} \phi \Vert_{0,E} $ and $\Vert \phi \Vert_{0,E} \simeq (\Vert b_E \phi \Vert_{0,E} + h_E\vert b_E \phi \vert_{1,E})$ for all $\phi \in \mathbb{P}_q(E)$. 

Let us set $w_E={\rm bulk}_E(u^\NN) b_E \in V$. Then, 
$$
\Vert {\rm bulk}_E(u^\NN) \Vert_{0,E}^2 \lesssim \int_E {\rm bulk}_E(u^\NN)^2 b_E = \int_E {\rm bulk}_E(u^\NN) \, w_E
$$
Writing
\begin{equation*}
\begin{split}
{\rm bulk}_E(u^\NN) &= (f-Lu^\NN) + \nabla(\Pi_{E,q} (\mu \nabla u^\NN)-\mu \nabla u^\NN) \\
& \  +  \  \Pi_{E,q-1}( \boldsymbol{\beta}\cdot \nabla u^\NN) - \boldsymbol{\beta}\cdot \nabla u^\NN 
\ + \   \Pi_{E,q-1}( \sigma u^\NN) -  \sigma u^\NN  \\
& \ + \ \Pi_{q-1,E}f-f \,,  \\
\end{split}
\end{equation*}
we obtain
\begin{equation*}
\begin{split}
\int_E {\rm bulk}_E(u^\NN) \, w_E &= a(u-u^\NN, w_E) - \int_E (\Pi_{E,q} (\mu \nabla u^\NN)-\mu \nabla u^\NN)\cdot \nabla w_E \\
& \  +  \  \int_E (\Pi_{E,q-1}( \boldsymbol{\beta}\cdot \nabla u^\NN)  - \boldsymbol{\beta}\cdot \nabla u^\NN )(w_E - m(w_E)) \\
&  \ + \   \int_E (\Pi_{E,q-1}( \sigma u^\NN) -  \sigma u^\NN) (w_E - m(w_E))  \\
& \ + \ \int_E (\Pi_{q-1,E}f-f)(w_E - m(w_E))  \,,  \\
\end{split}
\end{equation*}
whence
$$
\Vert {\rm bulk}_E(u^\NN) \Vert_{0,E}^2 \lesssim \left( \vert u - u^\NN \vert_{1,E} + \sum_{k=1}^3 \eta_{{\rm coef},k}(E) + \eta_{{\rm rhs},1}(E) \right) \vert w_E \vert_{1,E} \,.
$$
Using $ \vert w_E \vert_{1,E} \lesssim h_E^{-1} \Vert {\rm bulk}_E(u^\NN) \Vert_{0,E}$, we arrive at
\begin{equation}\label{eq:aposteriori4}
h_E \Vert {\rm bulk}_E(u^\NN) \Vert_{0,E} \lesssim \vert u - u^\NN \vert_{1,E} + \sum_{k=1}^3 \eta_{{\rm coef},k}(E) + \eta_{{\rm rhs},1}(E) \,.
\end{equation}

Let us now turn to the jump contribution to the estimator. Given an edge $e \subset \partial E$ shared with the element $E'$, we introduce a non-negative bubble function $b_e \in V$, with support in $E \cup E'$ and such that $\Vert \phi \Vert_{0,e} \simeq \Vert  b_e^{1/2} \phi \Vert_{0,e} $ and $ (h_E^{-1/2} \Vert b_e \phi \Vert_{0,E} + h_E^{1/2}\vert b_e \phi \vert_{1,E}) \lesssim \Vert \phi \Vert_{0,e}$ for all $\phi \in \mathbb{P}_q(E)$. 

Let us extend the function ${\rm jump_e(u^\NN)}$ onto $E \cup E'$ to be constant in the normal direction to $e$, obtaining a polynomial of degree $q$ in each element. Let us set $w_e = {\rm jump_e(u^\NN)} b_e \in V$. Then, writing $E_1=E$ and $E_2=E'$, one has
\begin{equation*}
\begin{split}
\Vert {\rm jump}_e(u^\NN) \Vert_{0,e}^2 &\lesssim \int_e {\rm jump}_e(u^\NN)^2 b_e = \int_e {\rm jump}_e(u^\NN) \, w_e \\
& \  = \ \int_e {\rm jump}_e(u^\NN - u) \, w_e \,, \\
& \  = \ \sum_{i=1}^2 \int_{E_i} \nabla \cdot [ (\Pi_{E_1,q}(\mu \nabla u^\NN) - \mu \nabla u) \, w_e ] \\
& \  = \ \sum_{i=1}^2 \int_{E_i}   [ \nabla \cdot \Pi_{E_1,q}(\mu \nabla u^\NN) - \nabla \cdot (\mu \nabla u) ]  w_e  \\
& \  \quad + \ \sum_{i=1}^2 \int_{E_i}   [ \Pi_{E_1,q}(\mu \nabla u^\NN) - \mu \nabla u ] \cdot \nabla  w_e \,.
\end{split}
\end{equation*}
We now recall that
$$
\nabla \cdot \Pi_{E_i,q} (\mu \nabla u^\NN)  = {\rm bulk}_{E_i}(u^\NN) - \Pi_{E_i,q-1}f + \Pi_{E_i,q-1}( \boldsymbol{\beta}\cdot \nabla u^\NN + \sigma u^\NN) \,,
$$
as well as $\nabla \cdot (\mu \nabla u) = - f + \boldsymbol{\beta}\cdot \nabla u + \sigma u$. We write $u= u^\NN + (u-u^\NN)$ and we proceed as in the proof of \eqref{eq:aposteriori4}, using now the bounds $ \Vert w_e \Vert_{0,E_i} \lesssim h_{E_i}^{1/2} \Vert {\rm jump}_e (u^\NN) \Vert_{0,e}$ and $ \vert w_e \vert_{1,E_i} \lesssim h_{E_i}^{-1/2} \Vert {\rm jump}_e (u^\NN) \Vert_{0,e}$, arriving at the bound
\begin{equation}\label{eq:aposteriori5}
\begin{split}
h_E^{1/2} \sum_{e \subset \partial E} \Vert \,{\rm jump}_e(u^\NN)   \, \Vert_{0,e}  &\lesssim \vert u - u^\NN \vert_{1,D_E} +
\sum_{E' \subset D_E} h_{E'} \Vert {\rm bulk}_{E'}(u^\NN) \Vert_{0,E'} \\
& \ \ \qquad + \sum_{E' \subset D_E} \left(
 \sum_{k=1}^3 \eta_{{\rm coef},k}(E') + \eta_{{\rm rhs},1}(E') \right) \,.
\end{split}
\end{equation}
Together with \eqref{eq:aposteriori4}, this gives the bound \eqref{eq:aposteriori3a}.
In order to derive \eqref{eq:aposteriori3b}, we write \eqref{eq:eta-local-loss} as
$$
C_h^{-1} \eta_{\rm loss}(E) = \left( \sum_{i \in I_h^E}  r_{h,i}^2(u^\NN) \right)^{1/2} = \quad
\sup_{\boldsymbol{v}} \frac1{\Vert \boldsymbol{v}\Vert_2} \sum_{i \in I_h^E}  r_{h,i}(u^\NN) v_i 
$$
where $\boldsymbol{v} = (v_i) \in \mathbb{R}^{{\rm card} I_h^E }$. Defining the function $v_h^E = \sum_{i \in I_h^E} v_i \varphi_i \in V_h$, which is supported in $D_E$, and recalling \eqref{eq:residuals}, we have
$$
\sum_{i \in I_h^E}  r_{h,i}(u^\NN) v_i  = F_h(v_h^E) - a_h(u^\NN,v_h^E)\,.
$$
By the left-hand inequality in \eqref{eq:norm-equiv-Vh}, we obtain
$$
\frac{c_h}{C_h} \, \eta_{{\rm loss}}(E) \ \leq \ \sup_{v_h^E} \frac{F_h(v_h^E) - a_h(u^\NN,v_h^E)}{\vert v_h^E \vert_{1,D_E}}\,. 
$$
Now we write
\begin{equation*}
\begin{split}
F_h(v_h^E) - a_h(u^\NN,v_h^E) & = F_h(v_h^E) - F(v_h^E) \\
& \quad + f(v_h^E) - a(u^\NN,v_h^E)  \\
& \quad + a(u^\NN,v_h^E) - a_h(u^\NN,v_h^E) \,.
\end{split}
\end{equation*}
The term $ F_h(v_h^E) - F(v_h^E) =[F_h(v_h^E) - F_\pi(v_h^E)] +  [F_\pi(v_h^E) - F(v_h^E)]$ can be bounded as done for the terms (I) and (VII) above, yielding
$$
\vert F_h(v_h^E) - F(v_h^E)  \vert \lesssim  \sum_{E' \subset D_E} \left(\eta_{{\rm rhs},1}(E')+ \eta_{{\rm rhs},2}(E') \right) \vert v_h^E \vert_{1,E'} \,.
$$
Similarly, the term $a(u^\NN,v_h^E) - a_h(u^\NN,v_h^E)$ can be handled as done for the terms (III) and (VIII) above, obtaining
$$
\vert a(u^\NN,v_h^E) - a_h(u^\NN,v_h^E)  \vert \lesssim  \sum_{E' \subset D_E} \left(\sum_{k=1}^6\eta_{{\rm coeff},k}(E') \right) \vert v_h^E \vert_{1,E'} \,.
$$
Finally, one has $\vert f(v_h^E) - a(u^\NN,v_h^E) \vert \lesssim \vert u-u^\NN \vert_{1,D_E} \vert v_h^E \vert_{1,D_E}$, thereby concluding the proof of \eqref{eq:aposteriori3b}.
\endproof

\section{Numerical results}\label{sec:numerics}


Let us consider the two-dimensional domain $\Omega=(0,1)^2$ and the Poisson problem:
\begin{equation}\label{eq:model-pb-poisson}
\begin{cases}
-\Delta u = f & \text{in \ } \Omega\,, \\
\ \ \, u=g & \text{on \ } \Gamma \,, \end{cases}
\end{equation}
with the functions $f$ and $g$ such that the exact solution, represented in Fig. \ref{fig:solution6}, is 
\begin{equation}\label{eq:sol6}
u(x,y) = \tanh\left[2\left(x^3 - y^4\right)\right]. 
\end{equation}
\begin{figure}[t!]
\centering 
\captionsetup{justification=centering}
  \includegraphics[width=0.75\linewidth]{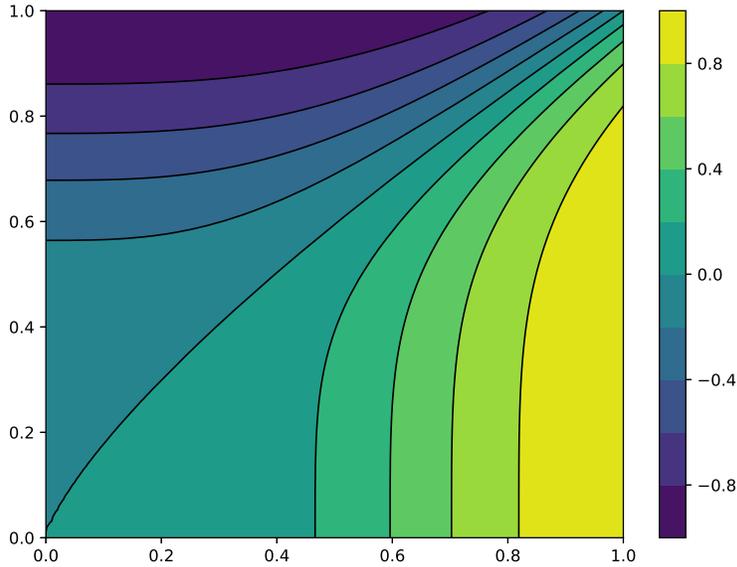} 
  \caption{Graphical representation of the exact solution $u(x,y)$ in \eqref{eq:sol6}}
  \label{fig:solution6}
\end{figure}
Problem \eqref{eq:model-pb-poisson} is numerically solved by the VPINN discretization described in Section \ref{sec:sub_discretization}, extended to handle non-homogeneous Dirichlet condition as mentioned in Remark \ref{rem:other-bcs}. The used VPINN is a feed-forward fully connected neural network comprised by an input layer with input dimension $n=2$, three hidden layers with 50 neurons each and an output layer with a single output variable; it thus contains 7851 trainable weights; furthermore, in all the layers except the output one the activation function is the hyperbolic tangent. The VPINN output is modified as described in \cite{BeCaPi2021} to exactly impose the Dirichlet boundary conditions. Gaussian quadrature rules of order $q=3$ are used in the definition of the loss function. 

For ease of implementation, the orthogonal projection operators $\Pi_{E,k}$, defined in Section \ref{sec:aposteriori-theory}, are mimiked by interpolation operators as follows. Let us initially consider the elemental Lagrange interpolation operator ${\cal I}_{E,k}:C^0(E)\rightarrow \mathbb P_k(E)$; then, to guarantee orthogonality to constants, the projection operator $\tilde{\Pi}_{E,k}:C^0(E)\rightarrow \mathbb P_k(E)$ is defined by setting 
\[
\tilde{\Pi}_{E,k}\varphi  := {\cal I}_{E,k}\varphi  + \dfrac{\int_E \left(\varphi - {\cal I}_{E,k}\varphi\right)}{\vert E\vert}, \hspace{0.5cm}\forall \varphi \in C^0(E),
\]
where, in practice, the integral $\int_E \left(\varphi - {\cal I}_{E,k}\varphi\right)$ can be computed with quadrature rules that are more accurate than the ones used in the other operations. In this work we use quadrature rules of order 7 in each element.

The VPINN is trained on different meshes and the corresponding error estimators $\left(\sum_{E \in {\cal T}_h}  \eta^2(E) \right)^{1/2}$ are computed. Once more, when exact integrals are involved, they are approximated with higher order quadrature rules. The obtained results are shown in Fig. \ref{fig:decays}, where the values of the $H^1$-error and the a posteriori estimator are displayed for several meshes of stepsize $h$.  Remarkably,  the error estimators  (red dots) behave very similarly to the corresponding energy errors (blue dots). Moreover, coherently with the results discussed in \cite{BeCaPi2021}, after an initial preasymptotic phase all dots are aligned on straight lines with slopes very close to 4 (the slope of the red line is 3.81, the slope of the blue line is 3.92).\\
\begin{figure}[t!]
\centering 
\captionsetup{justification=centering}
  \includegraphics[width=0.6\linewidth]{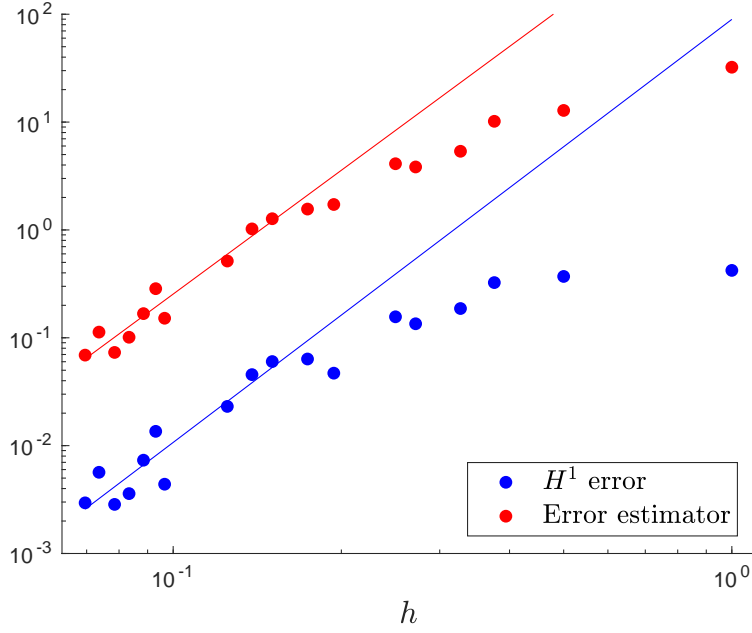} 
  \caption{$H^1$ errors (blue dots) obtained by training the same VPINN on different meshes,  and corresponding error estimators (red dots)}
  \label{fig:decays}
\end{figure}

It is also interesting to note that the terms appearing in the a posteriori estimator (recall \eqref{eq:aposteriori1}) exhibit different behaviors during the training of a single VPINN. This phenomenon is highlighted in Fig. \ref{fig:during_training}, where one can observe the evolution of the quantities $\eta_{\rm rhs}$, $\eta_{\rm coef}$, $\eta_{\rm res}$, $\eta_{\rm loss}$, $\eta$ and $\vert u - u^\NN \vert_{1,\Omega}$, where $\eta_. = \left(\sum_{E \in {\cal T}_h}  \eta_.^2(E) \right)^{1/2}$. It can be observed that, during this training, while the value of the loss function decreases, the accuracy remains almost constant because other sources of error, independent of the neural network, prevail. 
\begin{figure}[t!]
\centering 
\captionsetup{justification=centering}
  \includegraphics[width=0.6\linewidth]{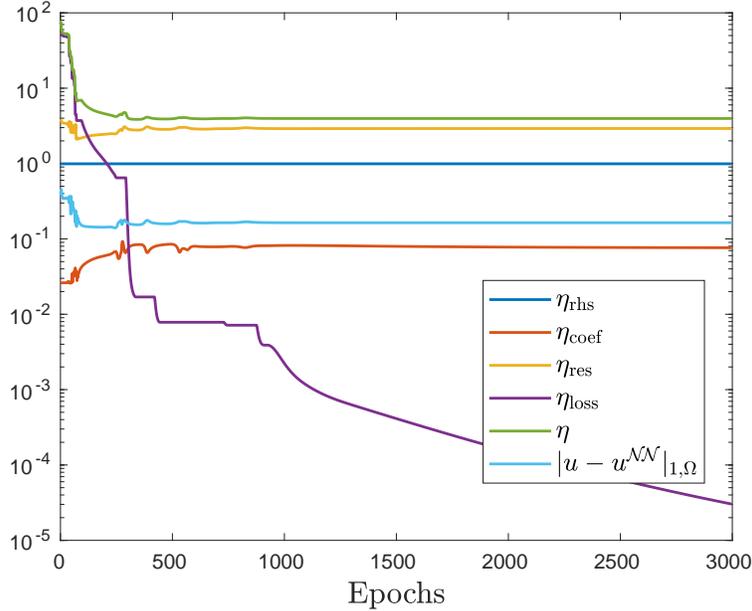} 
  \caption{Evolution of the addends of the error estimator $\eta$ during training}
  \label{fig:during_training}
\end{figure}

\section{Conclusions}\label{sec:conclusions}

We considered the discretization of a model elliptic boundary-value problem by variational physics-informed neural networks (VPINNs), in which test functions are continuous, piecewise linear functions on a triangulation of the domain. The scheme can be viewed as an instance of a least-square/Petrov-Galerkin method.

We introduced an a posteriori error estimator, which sums-up four contributions: the equation residual (measuring the elemental bulk residuals and the edge jump terms, for approximated coefficients and right-hand side), the coefficients' oscillation,  the right-hand side's oscillation, and a scaled value of the loss-function. The latter term corresponds to an inexact solve of the algebraic system arising from the discretization of the variational equations.

The main result of the paper is the proof that the estimator provides a global upper bound and a local lower bound for the energy norm of the error between the exact and VPINN solutions.  In other words, the a posteriori estimator is both reliable and efficient. 
Numerical results show an excellent agreement with the theoretical predictions.

In a forthcoming paper, we will investigate the use of the proposed estimator to design an adaptive strategy of discretization.

\bigskip
\noindent
{\bf Acknowledgements.}  The authors performed this research in the framework of the Italian MIUR Award ``Dipartimenti di Eccellenza 2018-2022" granted to the Department of Mathematical Sciences, Politecnico di Torino (CUP: E11G18000350001). The research leading to this paper has also been partially supported by the SmartData@PoliTO center for Big Data and Machine Learning technologies.
SB  was supported by the Italian MIUR PRIN Project 201744KLJL-004, CC was supported by the Italian MIUR PRIN Project 201752HKH8-003. 
The authors are members of the Italian INdAM-GNCS research group.

\bibliographystyle{siam}
\bibliography{bibliography}

\end{document}